\documentclass[11pt,reqno,oneside]{amsart}

\usepackage{amsmath, amsthm, amssymb}
\usepackage[applemac]{inputenc}

\usepackage{tensor}
\usepackage{bbm}%lowercase mathbb
\usepackage{geometry}
\usepackage{soul}%underline 
\usepackage[pdftex]{graphicx}
\usepackage{float}
\usepackage{subfig}
\usepackage[english]{babel}
\usepackage{tikz} % graphical language
\usepackage{setspace}
\usetikzlibrary{arrows,decorations.pathreplacing,decorations.markings,shapes.geometric}
\tikzset{
    bt/.style={draw=blue,thick},
    ns/.style={circle,draw=blue,fill=blue, inner sep=0pt, minimum size=2mm},
    string/.style={draw=#1, postaction={decorate}, decoration={markings,mark=at position .45 with {\arrow[blue]{triangle 60}}}},
    doublestring/.style={draw=#1, postaction={decorate}, decoration={markings, mark=at position .7 with {\arrow[blue]{triangle 60}}, 
    mark=at position .3 with {\arrowreversed[blue]{triangle 60}}}},
    costring/.style={draw=#1, postaction={decorate}, decoration={markings,mark=at position .55 with {\arrow[draw=#1]{<}}}},
    arr/.style={string=blue, thick},
    doublearr/.style={doublestring=blue, thick},
    lin/.style={blue},
    dlin/.style = {blue, dashed, thick},
    dot/.style={circle,draw=#1,fill=#1,inner sep=1pt},
}

\usepackage{dsfont}
\usepackage{color}
\usepackage[color,matrix,arrow]{xy}%diagrams
\graphicspath{{graphics/}}
\numberwithin{equation}{section}
\usepackage{fancyhdr}
\theoremstyle{definition}

\newtheorem{Prop}{Proposition}[section]
\newtheorem{Lemma}{Lemma}[section]
\newtheorem{Rem}{Remark}[section]

\newtheorem{Not}{Notation}[section]
\newtheorem{Conv}{Convention}[section]

\newcommand{\sh}[1]{\mathcal{#1}}

\DeclareMathOperator{\SL}{SL}

\DeclareMathOperator{\C}{\mathbb{C}}

\DeclareMathOperator{\Z}{\mathbb{Z}}

\DeclareMathOperator{\length}{\ell}

\DeclareMathOperator{\val}{\nu}

\DeclareMathOperator{\Gm}{\mathbb{G}_m}

\DeclareMathOperator{\Aff}{\mathbb{A}}

\DeclareMathOperator{\rot}{rot}

\DeclareMathOperator{\open}{o}
\DeclareMathOperator{\hyp}{h}
\usepackage[pdftex,
bookmarks=true,
bookmarksnumbered=true,
hypertexnames=false,
colorlinks = true,%colorlinks = true if you want color
linkcolor = black,
citecolor = black,
urlcolor = black]{hyperref}

\title{Orbits of subgroups of codimension one to four of the Iwahori group
in the affine flag variety of $\SL_2$}

\begin{document}

\author{Claude Eicher}
\address{SWITZERLAND}
\email{claudeeicher@gmail.com}
\date{\today}
\setcounter{tocdepth}{3}
\maketitle

\tableofcontents

\setstretch{1.4}

\begin{abstract}
We describe how each finite dimensional Schubert cell in the affine flag variety of $\SL_2$
decomposes into orbits for a chain of subgroups of codimension one to four of the Iwahori group. 
\end{abstract}

\section{Introduction}
In this article we describe how each finite dimensional Schubert cell in the affine flag variety of $\SL_2$ over $\C$
decomposes into orbits for the subgroups of the standard Iwahori group $I$ in the chain 

\begin{align}\label{eq:chainofsubgroupsofI}
I \supset I_{\{\alpha_1\}}
\supset I_{\{\alpha_1,\alpha_0\}} \supset I_{\{\alpha_1,\alpha_0,\delta+\alpha_1\}} \supset I_{\{\alpha_1,\alpha_0,\delta,\delta+\alpha_1\}}\;.
\end{align}

(Recall that in the affine flag variety the finite dimensional Schubert cells
are precisely the $I$-orbits.) Here $I_S \subset I$ is the subgroup obtained by
removing from $I$ the root subgroup associated to each element of $S$, a finite set of positive affine
roots of the corresponding affine Kac-Moody Lie algebra $\widehat{\mathfrak{sl}_2}$. As usual the simple roots are denoted by $\alpha_1$ and $\alpha_0$
and $\delta = \alpha_1+\alpha_0$. 
In the case of the last subgroup $I_{\{\alpha_1,\alpha_0,\delta,\delta+\alpha_1\}}$ we include the multiplicative group of loop rotations in
order for each Schubert cell to decompose into finitely many orbits. 
We find that for each containment in \eqref{eq:chainofsubgroupsofI} an orbit $O$
is either again an orbit for the subgroup or decomposes as $O = O^{\hyp} \sqcup O^{\open}$, where $O^{\hyp}$ is the locus of a hyperplane in natural coordinates of $O$ and $O^{\open}$ is its open complement. The natural 
coordinates are restrictions of the natural coordinates describing the Schubert cell as an affine space. 
If we label the orbits in this way, each $I_S$-orbit is named by the Schubert cell in which it is contained. 

\par
In the subsequent text we abbreviate the subgroups under consideration by $I_1 = I_{\{\alpha_1\}}$, 
$I_2 = I_{\{\alpha_1,\alpha_0\}}$, $I_3 = I_{\{\alpha_1,\alpha_0,\delta+\alpha_1\}}$, and $I_4 = I_{\{\alpha_1,\alpha_0,\delta,\delta+\alpha_1\}}$,
the index being the codimension in $I$. 

\subsection{Main result}
Our main result is Proposition \autoref{Prop:I4rotorbits}. 

\subsection{Related work}
 In the case of $\SL_2$ the corresponding $I_1$-orbits
 in the Kashiwara flag scheme and their closure relations are described without proof in \cite{Eic16b}[section 1.2].
The group $I_1$ has an obvious generalization to the subgroup $I \cap\; ^{s_i}I$ of the Iwahori group $I$ of the affine Kac-Moody group considered in \cite{Eic16b}, here $s_i$ is a simple reflection and $^{s_i}I$ is the $s_i$-conjugate of $I$. The subvariety $X_w \cap s_i X_w \cong \Gm \times \Aff^{\length(w)-1}$, 
where $X_w$ is a finite dimensional Schubert cell in the Kashiwara flag scheme and we require $s_iw < w$
in the Bruhat order for the element $w$ of the Weyl group, is an $I \cap\; ^{s_i}I$-orbit. In \cite{Eic16b} we investigate twisted $\sh{D}$-modules constructed from such a subvariety.  \par

In \cite{Eic20} we consider a subvariety of the affine flag variety of $\SL_2$
which is the $I_4^{\rot}$-orbit $E_1^{\open\open} \cong \Gm^2$ of the present work. 
We there investigate a construction
of twisted $\sh{D}$-modules based on a local system with two complex monodromy parameters
supported on $E_1^{\open\open}$. 
In the present work we exhibit $E_1^{\open\open}$ as the lowest dimensional member
of a collection of $I_4^{\rot}$-orbits isomorphic to $\Gm^2 \times \Aff^d$, $d \in \Z_{\geq 0}$. 
For all these orbits, a construction as in \cite{Eic20} would be possible. 

\subsection{Notation}
For an element $p \in \C((t))$ we denote by $p_n$ the coefficient of $t^n$. 
We use the degree $\nu_{\infty}: \C[t,t^{-1}]\setminus\{0\} \rightarrow \Z$
and the valuation $\nu: \C((t))\setminus\{0\} \rightarrow \Z$ defined
by $\nu_{\infty}(p) = N_1$ and $\nu(p)=N_0$ for $p = \sum_{n=N_0}^{N_1} p_n t^n$,
where $N_0 \leq N_1$, $p_{N_0} \neq 0$ and $p_{N_1} \neq 0$. 
If $p \in \C((t)) \setminus\{0\}$ we write $p = t^n p_{(0)}$ for a 
unique $n=\nu(p) \in \Z$ and a unique $p_{(0)} \in \C[[t]]^{\times}$, where $\C[[t]]^{\times}$
denotes the invertible elements of $\C[[t]]$. 
We always work with schemes and ind-varieties over $\C$ and only consider their $\C$-points. 

\section{Affine flag variety of $\SL_2$}
We consider the loop group $\SL_2((t))$ of $\SL_2$ over $\C$, given by

\begin{align*}
\SL_2((t)) = \left\{\begin{pmatrix} a & b \\ c & d\end{pmatrix}\; \vert\; a,b,c,d \in \C((t)),\; ad-bc = 1  \right\}\;.
\end{align*}

We consider its standard Iwahori subgroup $I$, a group scheme, of $\SL_2((t))$ given by

\begin{align*}
I = \left\{\begin{pmatrix} a & b \\ tc & d\end{pmatrix}\; \vert\; a,b,c,d \in \C[[t]],\; ad-t bc = 1  \right\}\;.
\end{align*}

Here, the condition $ad-tbc = 1$ obviously implies $a_0d_0=1$ and hence $a,d \in \C[[t]]^{\times}$. 
The quotient $\SL_2((t))/I$ is the affine flag variety of $\SL_2$, an ind-variety \cite{BD}[section 7.15.1]. 
The following lemma is well-known. 

\begin{Lemma} \label{Lemma:normalform}
For each $g \in \SL_2((t))$ there is a unique $n \in \Z$ and a unique $p \in \C[t,t^{-1}]$ such that either
\begin{align*}
g I = \begin{pmatrix} t^n & p \\0 & t^{-n}\end{pmatrix}I\quad \text{or}\quad g I = \begin{pmatrix} p & t^n \\ -t^{-n} & 0\end{pmatrix}I
\end{align*}
with $p=0$ or $\val_{\infty}(p) < n$ in the first case and $p=0$ or $\val_{\infty}(p) \leq n$
in the second case.
\end{Lemma}

\begin{Not}
We abbreviate this $I$-coset by $[n,p]$ and $[n,p]^{\prime}$ respectively.  
\end{Not}

The proof expresses $n$ and $p$ explicitly in terms of $g$.  
In the following we often apply the argument of the proof of this lemma to bring
an element of $\SL_2((t))/I$ to the form $[n,p]$ or $[n,p]^{\prime}$.

\begin{proof}

 \emph{Existence}. Let $g = \begin{pmatrix} a & b\\ c & d\end{pmatrix} \in \SL_2((t))$. 
 Then either $\frac{c}{d} \in t\C[[t]]$ or $\frac{d}{c} \in \C[[t]]$. If $d=0$ or $c=0$ this is clear,
hence assume $c \neq 0$ and $d \neq 0$. Then we write 
 $c=t^{\nu(c)} c_{(0)}$ and $d=t^{\nu(d)} d_{(0)}$. Then $\frac{c}{d} \in t^{\nu(c)-\nu(d)}\C[[t]]^{\times}$
 and hence we are in the first case if and only if $\nu(c)-\nu(d) \geq 1$. Since
 $\frac{d}{c} \in t^{\nu(d)-\nu(c)} \C[[t]]^{\times}$ we are in the second case if and only if $\nu(d)-\nu(c) \geq 0$. 
 \newline
\emph{Case $\frac{c}{d} \in t\C[[t]]$.} Then $d \neq 0$ and we have

\begin{align*}
g I &= g \begin{pmatrix} d_{(0)} & 0 \\ -t^{-\nu(d)}c & d_{(0)}^{-1}\end{pmatrix}I = \begin{pmatrix} t^{-\nu(d)} & b d_{(0)}^{-1} \\ 0 & t^{\nu(d)}\end{pmatrix}I = \begin{pmatrix} t^{-\nu(d)} & b d_{(0)}^{-1} \\ 0 & t^{\nu(d)}\end{pmatrix}\begin{pmatrix} 1 & s \\ 0 & 1\end{pmatrix}I \\
& = \begin{pmatrix} t^{-\nu(d)} & b d_{(0)}^{-1}+t^{-\nu(d)}s \\ 0 & t^{\nu(d)}\end{pmatrix}I 
\end{align*}

for $s \in \C[[t]]$ arbitrary, thus $gI$ is as in the first case indicated.\newline
\emph{Case $\frac{d}{c} \in \C[[t]]$.} Then $c \neq 0$ and we have

\begin{align*}
g I &= g \begin{pmatrix} -c_{(0)}^{-1} & t^{-\nu(c)}d \\ 0 & -c_{(0)}\end{pmatrix}I = \begin{pmatrix} -\frac{a}{c_{(0)}} & t^{-\nu(c)} \\ -t^{\nu(c)} & 0 \end{pmatrix}I = \begin{pmatrix} -\frac{a}{c_{(0)}} & t^{-\nu(c)} \\ -t^{\nu(c)} & 0 \end{pmatrix}\begin{pmatrix} 1 & 0 \\ t s & 1\end{pmatrix}I \\ 
& =\begin{pmatrix}-\frac{a}{c_{(0)}}+t^{-\nu(c)+1}s & t^{-\nu(c)} \\-t^{\nu(c)} & 0\end{pmatrix}I
\end{align*}
with $s \in \C[[t]]$ arbitrary, thus $gI$ is as in the second case indicated.\par

\emph{Uniqueness}. 
Consider the case when $g_1I=g_2I$ are in the first case. Thus we assume for some $g=\begin{pmatrix} a & b \\ tc & d\end{pmatrix} \in I$, $p, q \in \C[t,t^{-1}]$ with $\nu_{\infty}(p)
< n$, $\nu_{\infty}(q) < m$ that

\begin{align*}
\begin{pmatrix} t^n & p \\ 0 & t^{-n}\end{pmatrix} g= \begin{pmatrix} t^n a+tp c & t^n b+pd \\ t^{-n+1}c & t^{-n}d\end{pmatrix}= \begin{pmatrix} t^m & q \\ 0 & t^{-m}\end{pmatrix}\;.
\end{align*}

Thus $c = 0$, $a \in \C[[t]]^{\times}$ implies $a=1$ and $m=n$. Thus $d=1$. $t^n b+p = q$ then implies
$b=0$ and $p=q$. We have shown $g_1=g_2$. 
Next, consider the case when $g_1I=g_2 I$, $g_1 I$ in the second and $g_2 I$ in the first case. Thus we assume for some
$g = \begin{pmatrix} a & b \\ ct & d\end{pmatrix} \in I$, $p, q \in \C[t,t^{-1}]$ with $\val_{\infty}(p) \leq n$
and $\val_{\infty}(q) < m$
\begin{align*}
\begin{pmatrix} p & t^n \\-t^{-n} & 0\end{pmatrix}g = \begin{pmatrix} pa+ct^{1+n} & pb+t^n d \\-t^{-n} a & -t^{-n} b\end{pmatrix}= \begin{pmatrix} t^m & q \\0 & t^{-m}\end{pmatrix}\;.
\end{align*}
It follows $a=0$ and $c=t^{m-n-1}$ and $b=-t^{n-m}$. This contradicts $b,c \in \C[[t]]$. 
The final case when both $g_1I$ and $g_2I$ are in the second case is again similar. 
\end{proof}

\begin{Conv}\label{Conv:truncationinprimeandprimeprime}
Let $p \in \C((t))$. $[n,p]$ is understood to mean $[n, p^{(n)}]$, where
$p^{(n)} = \sum_{k < n} p_k t^k$. This is justified by the fact
that $\begin{pmatrix} t^n & p \\ 0 & t^{-n}\end{pmatrix}I = \begin{pmatrix} t^n & p^{(n)} \\ 0 & t^{-n}\end{pmatrix}I$ as used in the proof of the lemma.
Similarly, $[n,p]^{\prime}$
is understood to mean $[n, p^{(n+1)}]^{\prime}$. 
\end{Conv}

\begin{Rem}\label{Rem:Gmrotactionon[]primeand[]primeprime}
The $\Gm$-action on $\SL_2((t))$ given by loop rotation, $t \mapsto \gamma t$, $\gamma \in \C^{\times}$, 
is denoted by $\Gm^{\rot}$. It induces the same-named action on $\SL_2((t))/I$. It is given by $\gamma \cdot [n, p(t)] = [n, \gamma^n p(\gamma t)]$ and $\gamma \cdot [n,p(t)]^{\prime} = [n, \gamma^n p(\gamma t)]^{\prime}$, where $\gamma \in \C^{\times}$ and $p(t) \in \C((t))$. 
\end{Rem}

\begin{Rem}\label{Rem:involutiondots}
The automorphism of $\SL_2((t))/I$ given by translation by $\dot{s_1} = \begin{pmatrix} 0 & -1\\ 1 & 0\end{pmatrix}$ is given by $\dot{s_1}[n,0] = [-n,0]^{\prime}$ and $\dot{s_1} [n, p] = [-\nu(p), -\frac{t^{-n}}{p_{(0)}}]$
if $p \neq 0$ and 
and $\dot{s_1}[n,0]^{\prime} = [-n,0]$ and $\dot{s_1}[n, p]^{\prime} = [-\nu(p), -\frac{t^{-n}}{p_{(0)}}]^{\prime}$
if $p \neq 0$.  
Here we apply Convention \autoref{Conv:truncationinprimeandprimeprime}. 
This automorphism is an involution. 
\end{Rem}

\section{$I$-Orbits}

For $n \in \Z$ we define the $I$-orbit $E_n := I [n,0]$ and $O_n := I [n,0]^{\prime}$
in $\SL_2((t))/I$. We have chosen the notation $E$ and $O$
to indicate that these
orbits are even and odd dimensional respectively, as is apparent from Proposition \autoref{Prop:decompofFl} below.

\begin{Lemma}
Every point of $\SL_2((t))/I$ lies in $E_n$ or $O_n$ for some $n \in \Z$. 
\end{Lemma}

\begin{proof}
By Lemma \autoref{Lemma:normalform}, omitting the obvious case $p=0$, any point of $\SL_2((t))/I$
is either of the form $[n, p]$ for some $n \in \Z$
and $\nu_{\infty}(p) < n$ or $[n,p]^{\prime}$ for some $n \in \Z$
and $\nu_{\infty}(p) \leq n$.  
Writing $p=t^m p_{(0)}$ with $m=\nu(p) \in \Z$ and $p_{(0)} \in \C[[t]]^{\times} \cap \C[t]$ we
have $m < n$ and $0 \leq \nu_{\infty}(p_{(0)}) < n-m$ in the first case and $m \leq n$ and
$0 \leq \nu_{\infty}(p_{(0)}) \leq n-m$ in the second case. 
\newline

\emph{1. $[n,p] \in E_n$ for $m+n \geq 0$.}
We indeed have

\begin{align*}
\begin{pmatrix} 1 & -t^{m+n}p_{(0)} \\ 0 & 1\end{pmatrix} \begin{pmatrix} t^n & p \\ 0 & t^{-n}\end{pmatrix} = \begin{pmatrix} t^n & 0 \\ 0 & t^{-n}\end{pmatrix}\;.
\end{align*}
\newline

\emph{2. $[n,p] \in O_m$ for $m+n < 0$.} We indeed have
\begin{align*}
\begin{pmatrix} p_{(0)}^{-1} & 0 \\ -t^{-m-n} & p_{(0)}\end{pmatrix} \begin{pmatrix} t^n & p \\ 0 & t^{-n}\end{pmatrix}
\begin{pmatrix} 1 & 0 \\ -p_{(0)}^{-1}t^{n-m} & 1\end{pmatrix} = \begin{pmatrix} 0 & t^m \\ -t^{-m} & 0\end{pmatrix}\;.
\end{align*}
\newline

\emph{3. $[n, p]^{\prime} \in O_n$ for $m+n \geq 0$.} We indeed have

\begin{align*}
\begin{pmatrix}1 & p_{(0)} t^{m+n} \\ 0 & 1\end{pmatrix}\begin{pmatrix} p & t^n \\-t^{-n} & 0\end{pmatrix}
= \begin{pmatrix} 0 &  t^n \\-t^{-n} & 0\end{pmatrix}\;. 
\end{align*}
\newline

\emph{4. $[n, p]^{\prime} \in E_m$ for $m+n < 0$.} We indeed have

\begin{align*}
\begin{pmatrix}p_{(0)}^{-1} & 0 \\ t^{-m-n} & p_{(0)}\end{pmatrix} \begin{pmatrix} p & t^n \\ -t^{-n} & 0\end{pmatrix} 
\begin{pmatrix} 1 & -p_{(0)}^{-1}t^{n-m} \\ 0 & 1\end{pmatrix} = \begin{pmatrix} t^m & 0 \\0 & t^{-m}\end{pmatrix}\;. 
\end{align*}

\end{proof}

The following description of the $I$-orbits is fundamental for this article. 
\begin{Prop}\label{Prop:decompofFl} \noindent

\begin{enumerate}
\item For $n \geq 0$ we have
\begin{align}\label{eq:Lnpos}
E_n = [n, \C t^{-n}+\C t^{-n+1}+\dots + \C t^{n-1}]\;.
\end{align}
The subgroup of $I$ of elements
$\begin{pmatrix} 1 & p \\ 0 & 1\end{pmatrix}$
with $p \in \C[t]$ such that $\val_{\infty}(p) \leq 2n-1$ acts freely and transitively on $E_n$, so $E_n \cong \Aff^{2n}$. Any element $\begin{pmatrix} a & b \\ tc & d\end{pmatrix} \in I$ with $\nu(b) \geq 2n$ fixes
$[n,0]$.

\item For $n < 0$ we have 
\begin{align}
\begin{split}\label{eq:Lnneg}
E_n &= [-n-1, \C^{\times} t^n + \C t^{n+1}+ \dots + \C t^{-n-1}]^{\prime}\\
& \sqcup [-n-2, \C^{\times} t^n + \C t^{n+1}+\dots + \C t^{-n-2}]^{\prime}\sqcup \dots \sqcup [n, \C^{\times} t^n]^{\prime}\sqcup [n,0]\;.
\end{split}
\end{align} 
The subgroup of $I$ of elements $\begin{pmatrix}
1 & 0 \\ t p & 1\end{pmatrix}$ with $p \in \C[t]$  such that $\val_{\infty}(p) \leq -2n-1$
acts freely and transitively on $E_n$, so $E_n \cong \Aff^{-2n}$.   For $0 \leq k \leq -2n-1$ the map $p \mapsto 
\begin{pmatrix} 1 & 0 \\t p & 1 \end{pmatrix}[n,0]$ sends the $p$ with $\val_{\infty}(p) \leq -2n-1$ and
$\val(p) = k$ bijectively to $[-n-1-k, \C^{\times} t^n + \C t^{n+1}+\dots + \C t^{-n-1-k}]
^{\prime}$. Any element $\begin{pmatrix} a & b \\ tc & d\end{pmatrix} \in I$ with $\nu(c) \geq -2n$ fixes
$[n,0]$.

\item For $n \geq 0$ we have
\begin{align*}
O_n = [n, \C t^{-n}+\C t^{-n+1}+\dots + \C t^n]^{\prime}\;.
\end{align*}
The subgroup of $I$ of elements
$\begin{pmatrix} 1 & p \\¬†0 & 1\end{pmatrix}$
with $p \in \C[t]$ such that $\val_{\infty}(p) \leq 2n$ acts freely and transitively on $O_n$, so $O_n \cong \Aff^{2n+1}$. Any element $\begin{pmatrix} a & b \\ tc & d\end{pmatrix} \in I$ with $\nu(b) \geq 2n+1$ fixes
$[n,0]^{\prime}$. 

\item For $n < 0$ we have 
\begin{align*}
O_n &= [-n-1, \C^{\times} t^n + \C t^{n+1}+ \dots + \C t^{-n-2}]\\
& \sqcup [-n-2, \C^{\times} t^n + \C t^{n+1}+\dots + \C t^{-n-3}]\sqcup \dots \sqcup [n+1, \C^{\times} t^n]\sqcup [n,0]^{\prime}\;.
\end{align*} 
The subgroup of $I$ of elements $\begin{pmatrix}
1 & 0 \\ t p & 1\end{pmatrix}$ with $p \in \C[t]$  such that $\val_{\infty}(p) \leq -2n-2$
acts freely and transitively on $O_n$, so $O_n \cong \Aff^{-2n-1}$. For $0 \leq k \leq -2n-2$ the map $p \mapsto 
\begin{pmatrix} 1 & 0 \\t p & 1 \end{pmatrix}[n,0]^{\prime}$ sends the $p$ with $\val_{\infty}(p) \leq -2n-2$ and
$\val(p) = k$ bijectively to $[-n-1-k, \C^{\times} t^n + \C t^{n+1}+\dots + \C t^{-n-2-k}]$.
Any element $\begin{pmatrix} a & b \\ tc & d\end{pmatrix} \in I$ with $\nu(c) \geq -2n-1$ fixes
$[n,0]^{\prime}$. 
\end{enumerate}      

\end{Prop}

\begin{proof}\noindent

(1) For $\begin{pmatrix} a & b \\ tc & d\end{pmatrix} \in I$ we have

\begin{align*}
& \begin{pmatrix} a & b \\ tc & d\end{pmatrix} \begin{pmatrix} t^n & 0 \\ 0 & t^{-n}\end{pmatrix}I
= \begin{pmatrix} t^n a & t^{-n}b \\ t^{n+1}c & t^{-n}d\end{pmatrix}I = \begin{pmatrix} t^n & \frac{t^{-n}b}{d}\\ 0 & t^{-n}\end{pmatrix}I 
\end{align*}

since 

\begin{align*}
\frac{\widetilde{c}}{\widetilde{d}} = \frac{t^{2n+1}c}{d} \in t\C[[t]]
\end{align*}

and $\widetilde{d}=t^{-n}\widetilde{d}_{(0)}=t^{-n}d$. 

\begin{Not}
Here and in the following we often denote by 
$\widetilde{c}$ and $\widetilde{d}$ the entries of the bottom row of the element of $\SL_2((t))$ to which
we apply Lemma \autoref{Lemma:normalform}. 
\end{Not}

This shows $\subseteq$ in \eqref{eq:Lnpos} and that $[n,0]$ 
is fixed when $\nu(b) \geq 2n$. On the other hand 
we have

\begin{align*}
E_n = I[n,0] \supseteq \begin{pmatrix} 1 & \C+\C t+\dots + \C t^{2n-1} \\ 0 & 1\end{pmatrix} [n,0] = [n, \C t^{-n}+\dots + \C t^{n-1}]
\end{align*}

and this shows $\supseteq$ in \eqref{eq:Lnpos}. Also

\begin{align*}
\begin{pmatrix} 1 & p\\ 0 & 1\end{pmatrix} [n,0]= [n, t^{-n}p]
\end{align*}

shows that the indicated subgroup indeed acts freely and transitively on $E_n$. 
\par

(2) For $\begin{pmatrix} a & b \\ tc & d\end{pmatrix} \in I$ we have

\begin{align}\label{eq:eltofIon[n,0]prime}
\begin{pmatrix} a & b \\ tc & d\end{pmatrix} \begin{pmatrix} t^n & 0 \\ 0 & t^{-n}\end{pmatrix}I
= \begin{pmatrix} t^n a & t^{-n}b \\ t^{n+1}c & t^{-n}d\end{pmatrix}I\;.
\end{align}

We have

\begin{align*}
\frac{\widetilde{d}}{\widetilde{c}} = \frac{t^{-2n-1}d}{c} \in \C[[t]]
\end{align*}

if $-2n-1-\nu(c) \geq 0$ and then \eqref{eq:eltofIon[n,0]prime} equals

\begin{align*}
\begin{pmatrix} -\frac{t^n a}{c_{(0)}} & t^{-n-1-\nu(c)} \\ -t^{n+1+\nu(c)} & 0\end{pmatrix} I
\in [-n-1-\nu(c), \C^{\times} t^n+\C t^{n+1}+\dots + \C t^{-n-1-\nu(c)}]^{\prime}\;.
\end{align*}

As a special case we find for $\nu_{\infty}(c) \leq -2n-1$ that

\begin{align*}
& \begin{pmatrix} 1 & 0 \\ tc & 1\end{pmatrix} \begin{pmatrix} t^n & 0 \\ 0 & t^{-n}\end{pmatrix}I = 
\begin{pmatrix} -\frac{t^n}{c_{(0)}} & t^{-n-1-\nu(c)} \\ -t^{n+1+\nu(c)} & 0\end{pmatrix} I \\
& = \begin{pmatrix} p_n t^n+p_{n+1}t^{n+1}+\dots + p_{-n-1-\nu(c)} t^{-n-1-\nu(c)} & t^{-n-1-\nu(c)} \\ -t^{n+1+\nu(c)} & 0\end{pmatrix}I\;.
\end{align*}

Here $p_n \in \C^{\times}$ uniquely determines $c_{\nu(c)} \in \C^{\times}$, $p_{n+1} \in \C$ then uniquely determines $c_{\nu(c)+1} \in \C$,
and so on, until $p_{-n-1-\nu(c)} \in \C$ uniquely determines $c_{-2n-1} \in \C$. Thus $p$ uniquely determines $c$
and hence the statement of bijectivity and that the action is free and transitive follows. We have shown that for $0 \leq k \leq -2n-1$ 

\begin{align*}
E_n \supseteq [-n-1-k, \C^{\times} t^n+\C t^{n+1}+\dots+\C t^{-n-1-k}]^{\prime}
\end{align*}

and hence $\supseteq$ in \eqref{eq:Lnneg}. 
Assume now $\nu(c) \geq -2n$. Then \eqref{eq:eltofIon[n,0]prime}

equals

\begin{align*}
\begin{pmatrix} t^n & \frac{t^{-n}b}{d} \\ 0 & t^{-n}\end{pmatrix}I = \begin{pmatrix} t^n & 0 \\ 0 & t^{-n}\end{pmatrix}I
\end{align*}

since

\begin{align*}
\frac{\widetilde{c}}{\widetilde{d}} = \frac{t^{2n+1}c}{d} \in t\C[[t]]
\end{align*}

and $\frac{t^{-n}b}{d} \in t^{-n}\C[[t]] \subseteq t^n \C[[t]]$. Thus we have also shown $\subseteq$ in \eqref{eq:Lnneg}
and the statement that $[n,0]$ is fixed when $\nu(c) \geq -2n$.\par

(3) is analogous to (1). \par
(4) is analogous to (2). 
\end{proof}

\section{$I_1$-Orbits}
We have

\begin{align*}
I_1 &= \left\{\begin{pmatrix} a & tb \\ tc & d\end{pmatrix}\vert\;  a,b,c,d \in \C[[t]],\;  ad-t^2bc = 1\right\} \\
I_2 &= \left\{\begin{pmatrix} a & tb \\ t^2c & d\end{pmatrix}\; \vert\; a,b,c,d \in \C[[t]],\; ad-t^3bc = 1  \right\} \\
I_3 &= \left\{\begin{pmatrix} a & t^2b \\ t^2c & d\end{pmatrix}\; \vert\; a,b,c,d \in \C[[t]],\; ad-t^4bc = 1  \right\} \\
 I_4 &= \\
& \left\{\begin{pmatrix} \alpha(1+t^2a) & t^2b \\ t^2c & \alpha^{-1}(1+t^2d)\end{pmatrix}\;\vert\; \alpha \in \C^{\times}, a, b, c, d \in \C[[t]],\; a+d+t^2(ad-bc)=0  \right\}
\end{align*}

We use the parametrization of elements of $I_4$ given here in the computations below. 

\begin{Conv}
In the propositions below describing the decomposition of an orbit into orbits of a subgroup, we 
omit the trivial case of point orbits.
\end{Conv}

Each $I$-orbit is either an $I_1$-orbit or decomposes into two $I_1$-orbits as follows.  

\begin{Prop}\label{Prop:I1orbits} \noindent
\begin{enumerate}
\item $E_n$, $n < 0$, is
an $I_1$-orbit.
\item 
$O_n$, $n <0$, is an $I_1$-orbit.  

\item $E_n$, $n > 0$, decomposes into
the $I_1$-orbits
\begin{align*}
& E_n^{\open} := [n, \C^{\times}¬†t^{-n}+\C t^{-n+1}+\dots + \C t^{n-1}
] \cong \Gm \times \Aff^{2n-1}\\
& E_n^{\hyp} := [n, \C t^{-n+1}+\C t^{-n+2}+\dots + \C t^{n-1}]
\cong \Aff^{2n-1}\;.
\end{align*}

In fact, we have $I_3 [n,t^{-n}] \supseteq E_n^{\open}$ and $I_2 [n,0] \supseteq E_n^{\hyp}$. 

\item $O_n$, $n \geq 0$, decomposes into the
$I_1$-orbits
\begin{align*}
& O_n^{\open} := [n,\C^{\times} t^{-n}+\C t^{-n+1}+\dots + \C t^n]^{\prime}
\cong \Gm \times \Aff^{2n}\\
& O_n^{\hyp} := [n, \C t^{-n+1}+ \C t^{-n+2}+\dots + \C t^n]^{\prime}
\cong \Aff^{2n}\;. 
\end{align*}

In fact, we have $I_3[n,t^{-n}]^{\prime} \supseteq O_n^{\open}$ and $I_2[n,0]^{\prime} \supseteq O_n^{\hyp}$. 

\end{enumerate}
\end{Prop}

\begin{Rem}
By the proposition, the list of $I_1$-orbits in $\SL_2((t))/I$ with their distinguished points is

\begin{align*}
&E_n \ni [n,0],\ n \leq 0\\
& O_n \ni [n,0]^{\prime},\ n < 0\\
& E_n^{\open} \ni [n,t^{-n}],\ n > 0 \\
& E_n^{\hyp} \ni [n,0],\ n > 0\\
& O_n^{\open} \ni [n,t^{-n}]^{\prime},\ n \geq 0\\
& O_n^{\hyp} \ni [n,0]^{\prime},\ n \geq 0\;.
\end{align*}

The point orbits are $E_0=[0,0]$, $O^{\hyp}_0=[0,0]^{\prime}$. 
\end{Rem}

Here and below we call a choice of point that has ``simple'' coordinates 
in an orbit for the given group a distinguished point of the orbit. If this point lies in the orbit
of the subgroup, we also require it to be the distinguished point of that orbit.
We often act by a group element on the distinguished point in order to analyze the
group action in the orbit. 

\begin{Rem}
Since $\dot{s_1} I_1 \dot{s_1}^{-1} = I_1$ it follows that if 
$O$ is an $I_1$-orbit in $\SL_2((t))/I$, then so is $\dot{s_1}O$. From Remark \autoref{Rem:involutiondots} we find that the involution $\dot{s_1}$ of 
$\SL_2((t))/I$ restricts to isomorphisms

\begin{align*}
& E_n \xrightarrow{\cong} O_{-n}^{\hyp},\; n \leq 0 \\
& O_n \xrightarrow{\cong} E_{-n}^{\hyp},\; n < 0\\
& E_n^{\open} \xrightarrow{\cong} E_n^{\open},\; n  > 0 \\
& O_n^{\open} \xrightarrow{\cong} O_n^{\open},\; n \geq 0\;.
\end{align*}

\end{Rem}

\begin{proof}\noindent

(1) is a direct consequence of Proposition \autoref{Prop:decompofFl}(2) since $\begin{pmatrix} 1 & 0 \\ tp & 1\end{pmatrix} \in I_1$ for $p \in \C[t]$, $\nu_{\infty}(p) \leq 2n-1$.  \par

(2) is a direct consequence of Proposition \autoref{Prop:decompofFl}(4) since $\begin{pmatrix} 1 & 0 \\ tp & 1\end{pmatrix} \in I_1$ for $p \in \C[t]$, $\nu_{\infty}(p) \leq 2n-2$.\par

(3) By \eqref{eq:Lnpos} we have $E_n = E_n^{\open}\sqcup E_n^{\hyp}$. \par
\emph{$E_n^{\open}$ is an $I_1$-orbit and $I_3[n,t^{-n}]\supseteq E_n^{\open}$.} First, we show $I_1 [n, t^{-n}] \subseteq E_n^{\open}$. We have
for $\begin{pmatrix} a & tb \\ tc & d\end{pmatrix} \in I_1$

\begin{align}\label{eq:generaleltofI1on[n,t^(-n)]^prime}
\begin{pmatrix} a & tb \\ tc & d\end{pmatrix}\begin{pmatrix} t^n & t^{-n} \\ 0 & t^{-n}\end{pmatrix}I
= \begin{pmatrix} t^n a & t^{-n}a+t^{-n+1}b \\ t^{n+1}c & t^{-n+1}c+t^{-n}d\end{pmatrix}I\;.
\end{align}

Since 

\begin{align*}
\frac{\widetilde{c}}{\widetilde{d}} = \frac{t^{2n+1}c}{tc+d} \in t\C[[t]]
\end{align*}

and $\widetilde{d}=t^{-n}\widetilde{d}_{(0)} = t^{-n}(d+tc)$ \eqref{eq:generaleltofI1on[n,t^(-n)]^prime} equals

\begin{align*}
\begin{pmatrix} t^n & \frac{t^{-n}a+t^{-n+1}b}{d+tc} \\ 0 & t^{-n}\end{pmatrix}I \in E_n^{\open}\;.
\end{align*}

Second, we show $I_3 [n, t^{-n}] \supseteq E_n^{\open}$.
We have $\begin{pmatrix} \alpha(1+at) & 0 \\ 0 & \alpha^{-1}(1+at)^{-1}\end{pmatrix} \in I_3$  for $\alpha \in \C^{\times}$, $a \in \C$ 

\begin{align*}
\begin{pmatrix} \alpha(1+at) & 0 \\ 0 & \alpha^{-1}(1+at)^{-1}\end{pmatrix} \begin{pmatrix} t^n & t^{-n} \\ 0 & t^{-n}\end{pmatrix} I = \begin{pmatrix} t^n & \alpha^2 t^{-n}(1+2at+a^2t^2) \\ 0 & t^{-n}\end{pmatrix}I
\end{align*}

and hence the coefficient of $t^{-n}$ and $t^{-n+1}$ can be made arbitrary elements
of $\C^{\times}$ and $\C$. Acting
with $\begin{pmatrix} 1 & t^2b \\ 0 & 1\end{pmatrix} \in I_3$ for $b \in \C[[t]]$ the coefficients of $t^{-n+2}$, 
$\dots$, $t^{n-1}$ can also be made arbitrary. 

\emph{$E_n^{\hyp}$ is an $I_1$-orbit and $I_2[n,0]\supseteq E_n^{\hyp}$.} First, since $E_n$ 
and $E_n^{\open}$ are $I_1$-invariant, so is $E_n^{\hyp}$.  Second, we have $I_2[n,0]\supseteq E_n^{\hyp}$.
Indeed

\begin{align*}
I_2 [n,0] \supseteq \begin{pmatrix} 1 & \C t+\C t^2+\dots +\C t^{2n-1} \\ 0 & 1\end{pmatrix} [n,0]= E_n^{\hyp}\;.
\end{align*}

\par

(4) is analogous to (3). 
\end{proof}

\section{$I_2$-Orbits}
Each $I_1$-orbit, recall that some of them are in fact $I$-orbits, 
is either an $I_2$-orbit or decomposes into two $I_2$-orbits as follows. 
\begin{Prop}\label{Prop:I2orbits}
\noindent
\begin{enumerate}
\item $E_n^{\open}$, $E_n^{\hyp}$, $n > 0$, is an $I_2$-orbit. 

\item $O_n^{\open}$, $n \geq 0$, $O_n^{\hyp}$, $n > 0$,  is an $I_2$-orbit.  

\item $E_n$, $n < 0$, decomposes into the $I_2$-orbits

\begin{align*}
E_n^{\open} &:= [-n-1, \C^{\times}t^n+\C t^{n+1}+\dots +\C t^{-n-1}]^{\prime} \cong \Gm \times \Aff^{-2n-1}\\
E_n^{\hyp} &:= [-n-2, \C^{\times} t^n + \C t^{n+1}+\dots + \C t^{-n-2}]^{\prime}\sqcup \dots \sqcup [n, \C^{\times} t^n]^{\prime}\sqcup [n,0]\cong \Aff^{-2n-1}\;.
\end{align*}

In fact, we have $I_4[-n-1,t^n]^{\prime} \supseteq E_n^{\open}$
and $I_4[n,0] \supseteq E_n^{\hyp}$. 

\item $O_n$, $n < 0$, decomposes into the $I_2$-orbits

\begin{align*}
O_n^{\open} &:= [-n-1, \C^{\times}t^n+\C t^{n+1}+\dots+\C t^{-n-2}] \cong \Gm \times \Aff^{-2n-2} \\
O_n^{\hyp} &:= [-n-2,\C^{\times}t^n +\C t^{n+1}+\dots + \C t^{-n-3}]\sqcup \dots \sqcup [n+1, \C^{\times} t^n]\sqcup [n,0]^{\prime} \cong \Aff^{-2n-2}\;.
\end{align*}

In fact, we have $I_4[-n-1,t^n] \supseteq O_n^{\open}$ and $I_4[n,0]^{\prime} \supseteq O_n^{\hyp}$. 

\end{enumerate}
\end{Prop}

\begin{Rem} By the proposition, the list of $I_2$-orbits in $\SL_2((t))/I$ with their distinguished points is

\begin{align*}
& E_0 = [0,0] \\
& E_n^{\open} \ni \begin{cases} [n, t^{-n}] & n > 0 \\ [-n-1,t^n]^{\prime} & n <0\end{cases},\ n  \in \Z\setminus\{0\} \\
& E_n^{\hyp} \ni [n,0],\ n \in \Z\setminus\{0\}\\
& O_n^{\open} \ni \begin{cases} [n,t^{-n}]^{\prime} & n \geq 0 \\ [-n-1,t^n] & n < 0\end{cases},\ n \in \Z\\
& O_n^{\hyp} \ni [n,0]^{\prime},\ n \in \Z\;.
\end{align*}

The point orbits are $E_0= [0,0]$, $O_{-1}^{\hyp}=[-1,0]^{\prime}$, $O_0^{\hyp}=[0,0]^{\prime}$. 
\end{Rem}

\begin{proof}
\noindent

(1) follows from Proposition \autoref{Prop:I1orbits}(3). 

\par

(2) follows from Proposition \autoref{Prop:I1orbits}(4). \par

(3) By \eqref{eq:Lnneg} we have $E_n = E_n^{\open} \sqcup E_n^{\hyp}$. \par
\emph{$E_n^{\open}$ is an $I_2$-orbit and $I_4[-n-1,t^n]^{\prime} \supseteq E_n^{\open}$.} 
First, we show $I_2[-n-1, t^n]^{\prime} \subseteq E_n^{\open}$. We have for $\begin{pmatrix} a & tb \\ t^2c & d\end{pmatrix} \in I_2$

\begin{align}\label{eq:eltofI2on[-n-1,t^n]^primenew}
\begin{pmatrix} a & tb \\ t^2c & d\end{pmatrix} \begin{pmatrix} t^n & t^{-n-1} \\ -t^{n+1} & 0\end{pmatrix}I
= \begin{pmatrix} t^n a-t^{n+2}b & t^{-n-1}a \\ t^{n+2}c-t^{n+1}d & t^{-n+1}c\end{pmatrix} I\;. 
\end{align}

We have

\begin{align*}
\frac{\widetilde{d}}{\widetilde{c}} = \frac{t^{-n+1}c}{t^{n+2}c-t^{n+1}d} = \frac{t^{-2n}c}{tc-d} \in \C[[t]]\;.
\end{align*}

Hence, since $\widetilde{c}=t^{n+1}\widetilde{c}_{(0)}=t^{n+1}(-d+tc)$, \eqref{eq:eltofI2on[-n-1,t^n]^primenew} equals

\begin{align*}
\begin{pmatrix} \frac{t^n a-t^{n+2}b}{d-tc} & t^{-n-1} \\ -t^{n+1} & 0\end{pmatrix}I \in E_n^{\open}\;.
\end{align*}

Second, we show $I_4 [-n-1,t^n]^{\prime} \supseteq E_n^{\open}$. 
For $\alpha \in \C^{\times}$, $a, c \in \C$ we have

\begin{align*}
\begin{pmatrix} \alpha(1+t^2a) & 0 \\ t^2 c & \alpha^{-1}(1+t^2a)^{-1}\end{pmatrix} \in I_4
\end{align*}

and 

\begin{align}\label{eq:someeltofI4on[-n-1,n]primeprime}
\begin{split}
& \begin{pmatrix} \alpha(1+t^2a) & 0 \\ t^2 c & \alpha^{-1}(1+t^2a)^{-1}\end{pmatrix}\begin{pmatrix} t^n & t^{-n-1} \\ -t^{n+1} & 0\end{pmatrix}I = \\
& = \begin{pmatrix} \alpha(1+t^2a)t^n & \alpha(1+t^2a)t^{-n-1} \\ t^{n+2}c-\alpha^{-1}(1+t^2a)^{-1}t^{n+1} & t^{-n+1}c\end{pmatrix}I\;.
\end{split}
\end{align}

We have 

\begin{align*}
\frac{\widetilde{d}}{\widetilde{c}} = \frac{t^{-2n}c}{tc-\alpha^{-1}(1+t^2a)^{-1}} \in \C[[t]] 
\end{align*}

and since 

\begin{align*}
\widetilde{c} = t^{n+1}\widetilde{c}_{(0)} = t^{n+1}(tc-\alpha^{-1}(1+t^2a)^{-1})
\end{align*}

\eqref{eq:someeltofI4on[-n-1,n]primeprime} equals

\begin{align*}
\begin{pmatrix} \frac{\alpha^2(1+t^2 a)^2 t^n}{1-\alpha(1+t^2a)tc}& t^{-n-1} \\ -t^{n+1} & 0\end{pmatrix} I\;.  
\end{align*}

Now

\begin{align*}
& \frac{\alpha^2(1+t^2 a)^2 t^n}{1-\alpha(1+t^2a)tc} \in \\
&\alpha^2(t^n+\alpha c t^{n+1}+(2a+\alpha^2 c^2)t^{n+2})+p_{n+3}t^{n+3}+\dots +p_{-n-1}t^{-n-1}+t^{-n}\C[[t]]
\end{align*}

and hence the coefficients of $t^n$, $t^{n+1}$, $t^{n+2}$ are arbitrary by choice of $\alpha, a, c$. We have $\begin{pmatrix} 1 & t^2b \\ 0 & 1\end{pmatrix} \in  I_4$ for $b \in \C[[t]]$ and

\begin{align*}
\begin{pmatrix} 1 & t^2b \\ 0 & 1\end{pmatrix}[-n-1,p]^{\prime} = [-n-1,p-t^{n+3}b]^{\prime}
\end{align*}

and hence the coefficients $p_{n+3}$, \dots, $p_{-n-1}$ can also be made arbitrary elements of $\C$. 
\par

\emph{$E_n^{\hyp}$ is an $I_2$-orbit
and $I_4[n,0] \supseteq E_n^{\hyp}$.} 
Since $E_n = E_n^{\open}\sqcup E_n^{\hyp}$ and $E_n^{\open}$ are $I_2$-invariant, so is $E_n^{\hyp}$. We have by 
Proposition \autoref{Prop:decompofFl}(2) for $0 \leq k \leq -2n-2$

\begin{align*}
& I_4[n,0] \supseteq \begin{pmatrix} 1 & 0 \\ t^2(\C^{\times}t^k+\C t^{k+1}+\dots + \C t^{-2n-2}) & 1\end{pmatrix} [n,0]  \\
&  = [-n-2-k, \C^{\times}t^n+\C t^{n+1}+\dots + \C t^{-n-2-k}]^{\prime}
\end{align*}

and hence $I_4[n,0] \supseteq E_n^{\hyp}$.  

\par

\emph{(4)} is analogous to (3). 

\end{proof}

\section{$I_3$-Orbits}
Each $I_2$-orbit, recall that some of them are in fact $I_1$-orbits,
is either an $I_3$-orbit or decomposes into two $I_3$-orbits as follows.

\begin{Prop}\label{Prop:I3orbits}\noindent
\begin{enumerate}
\item $E_n^{\open}$, $n \neq 0$, $E_n^{\hyp}$, $n < 0$, is an $I_3$-orbit. 

\item $O_n^{\open}$, $n \in \Z$, $O_n^{\hyp}$, $n < 0$, is an $I_3$-orbit. 

\item $E_n^{\hyp}$, $n > 0$, decomposes into the $I_3$-orbits

\begin{align*}
E_n^{\hyp\open} &:= [n, \C^{\times}t^{-n+1}+\C t^{-n+2}+\dots  + \C t^{n-1}] \cong \Gm \times \Aff^{2n-2} \\
E_n^{\hyp\hyp} &:= [n, \C t^{-n+2}+\dots+\C t^{n-1}] \cong \Aff^{2n-2}\;.
\end{align*}

In fact, we have $I_4 [n,t^{-n+1}] \supseteq E_n^{\hyp\open}$ 
and $I_4 [n,0] \supseteq E_n^{\hyp\hyp}$. 

\item $O_n^{\hyp}$, $n > 0$, decomposes into the $I_3$-orbits

\begin{align*}
O_n^{\hyp\open} &:= [n, \C^{\times}t^{-n+1}+\C t^{-n+2}+\dots + \C t^n]^{\prime} \cong \Gm \times \Aff^{2n-1}\\
O_n^{\hyp\hyp} &:= [n, \C t^{-n+2}+\dots + \C t^n]^{\prime} \cong \Aff^{2n-1}\;.
\end{align*}

In fact, we have $I_4[n,t^{-n+1}]^{\prime} \supseteq O_n^{\hyp\open}$
and $I_4[n,0]^{\prime} \supseteq O_n^{\hyp\hyp}$. 

\end{enumerate}
\end{Prop}

\begin{Rem} By the proposition, the list of $I_3$-orbits in $\SL_2((t))/I$ with their distinguished points is

\begin{align*}
& E_0 = [0,0] \\
& E_n^{\open} \ni \begin{cases} [n, t^{-n}] & n > 0 \\ [-n-1,t^n]^{\prime} & n <0\end{cases},\ n  \in \Z\setminus\{0\} \\
& E_n^{\hyp} \ni [n,0],\ n < 0\\
& O_n^{\open} \ni \begin{cases} [n,t^{-n}]^{\prime} & n \geq 0 \\ [-n-1,t^n] & n < 0\end{cases},\ n \in \Z\\
& O_n^{\hyp} \ni [n,0]^{\prime},\ n \leq 0 \\
& E_n^{\hyp\open} \ni [n, t^{-n+1}],\ n > 0 \\
& E_n^{\hyp\hyp} \ni [n,0],\ n > 0 \\
& O_n^{\hyp\open} \ni [n,t^{-n+1}]^{\prime},\ n > 0 \\
& O_n^{\hyp\hyp} \ni [n, 0]^{\prime},\ n > 0\;.
\end{align*}

The point orbits are $E_0= [0,0]$, $O_{-1}^{\hyp}=[-1,0]^{\prime}$, $O_0^{\hyp}=[0,0]^{\prime}$, $E_1^{\hyp\hyp} = [1,0]$. 
\end{Rem}

\begin{Rem}
Since $\dot{s_1} I_3 \dot{s_1}^{-1} = I_3$ it follows that if 
$O$ is an $I_3$-orbit in $\SL_2((t))/I$, then so is $\dot{s_1}O$.
From Remark \autoref{Rem:involutiondots} we find that the involution $\dot{s_1}$ of 
$\SL_2((t))/I$ restricts to isomorphisms

\begin{align*}
& E_0=[0,0] \mapsto O_0^{\hyp} = [0,0]^{\prime}\\
& E_n^{\open} \xrightarrow{\cong} E_n^{\open},\ n > 0\\
& E_n^{\open} \xrightarrow{\cong} O_{-n}^{\hyp\open},\ n < 0 \\
& E_n^{\hyp} \xrightarrow{\cong} O_{-n}^{\hyp\hyp},\ n < 0 \\
& O_n^{\open} \xrightarrow{\cong} O_n^{\open},\ n \geq 0 \\
& O_n^{\open} \xrightarrow{\cong} E^{\hyp\open}_{-n},\ n < 0\\
& O_n^{\hyp} \xrightarrow{\cong} E^{\hyp\hyp}_{-n},\ n < 0\;.
\end{align*}

\end{Rem}

\begin{proof}\noindent

(1) \emph{$E_n^{\open}$, $n > 0$, is an $I_3$-orbit.} 
This follows from Proposition \autoref{Prop:I1orbits}(3).

\par
\emph{$E_n^{\open}$, $E_n^{\hyp}$, $n < 0$, is an $I_3$-orbit.} This follows from Proposition \autoref{Prop:I2orbits}(3).
\par

(2) \emph{$O_n^{\open}$, $n \geq 0$, is an $I_3$-orbit.} This follows from Proposition \autoref{Prop:I1orbits}(4). \par

\emph{$O_n^{\open}$, $O_n^{\hyp}$, $n < 0$, is an $I_3$-orbit.} 
This follows from Proposition \autoref{Prop:I2orbits}(4). \par

(3) \emph{$E_n^{\hyp\open}$ is an $I_3$-orbit and $I_4[n,t^{-n+1}] \supseteq E_n^{\hyp\open}$.} 
First, we show $I_3 [n,t^{-n+1}] \subseteq E_n^{\hyp\open}$. 
We have for $\begin{pmatrix} a & t^2b \\ t^2c & d\end{pmatrix} \in I_3$

\begin{align}\label{eq:eltofI3on[n,t^(-n+1)]prime}
\begin{pmatrix} a & t^2b \\ t^2 c & d\end{pmatrix} \begin{pmatrix} t^n & t^{-n+1} \\ 0 & t^{-n}\end{pmatrix}I = \begin{pmatrix} t^n a & t^{-n+1}a+t^{-n+2}b \\ t^{n+2}c & t^{-n+3}c+t^{-n}d\end{pmatrix}I\;.
\end{align}

We have

\begin{align*}
\frac{\widetilde{c}}{\widetilde{d}} = \frac{t^{2n+2}c}{t^3c+d} \in t\C[[t]]
\end{align*}

and $\widetilde{d}=t^{-n}\widetilde{d}_{(0)}=t^{-n}(d+t^3c)$ and hence \eqref{eq:eltofI3on[n,t^(-n+1)]prime} equals

\begin{align*}
\begin{pmatrix} t^n & \frac{t^{-n+1}a+t^{-n+2}b}{d+t^3c} \\ 0 & t^{-n}\end{pmatrix}I \in E_n^{\hyp\open}\;.
\end{align*}

Second, we show  $I_4[n,t^{-n+1}] \supseteq E_n^{\hyp\open}$. 
We have $\begin{pmatrix} \alpha & 0 \\ 0 & \alpha^{-1}\end{pmatrix} \in I_4$, $\alpha \in \C^{\times}$, and

\begin{align*}
\begin{pmatrix} \alpha & 0 \\ 0 & \alpha^{-1}\end{pmatrix} [n, t^{-n+1}] = [n, \alpha^2 t^{-n+1}]\;.
\end{align*}

Using that $\begin{pmatrix} 1 & t^2b \\ 0 & 1\end{pmatrix} \in I_4$ for $b \in \C[[t]]$ and

\begin{align*}
\begin{pmatrix} 1 & t^2b \\ 0 & 1\end{pmatrix} [n, p] = [n, p+t^{-n+2}b]
\end{align*}

we can make $p_{-n+2}$, \dots, $p_{n-1}$ arbitrary elements of $\C$. 
Thus we have shown $I_4 [n,t^{-n+1}] \supseteq E_n^{\hyp\open}$.
\par

\emph{$E_n^{\hyp\hyp}$ is an $I_3$-orbit and $I_4[n,0] \supseteq E_n^{\hyp\hyp}$.} $E_n^{\hyp\hyp}$ is  $I_3$-invariant since $E_n^{\hyp}$
and $E_n^{\hyp\open}$ are so. Since 

\begin{align*}
I_4 [n, 0] \supseteq \begin{pmatrix} 1 & \C t^2+\dots + \C t^{2n-1} \\
0 & 1\end{pmatrix} [n, 0] = [n, \C t^{-n+2}+\dots + \C t^{n-1}]
\end{align*}

it also follows that $I_4 [n, 0] \supseteq E_n^{\hyp\hyp}$. 

\par

(4) is analogous to (3). 

\end{proof}

\section{$I_4^{\rot}$-Orbits}

We can form the semidirect product $I_4^{\rot} = I_4 \rtimes \Gm^{\rot}$, $\rot$ indicating that we consider
the action by loop rotation of Remark \autoref{Rem:Gmrotactionon[]primeand[]primeprime}. $I_4^{\rot}$ naturally acts on $\SL_2((t))/I$. 
Each $I_3$-orbit, recall that some of them are in fact $I_2$- or $I_1$-orbits,
is either an $I_4^{\rot}$-orbit or decomposes into two $I_4^{\rot}$-orbits as follows. 

\begin{Prop}\label{Prop:I4rotorbits}
\noindent
\begin{enumerate}
\item $E_n^{\open}$, $E_n^{\hyp}$, $n < 0$, is an
$I_4$-orbit. 

\item $O_n^{\open}$, $n \leq 0$, $O_n^{\hyp}$, $n \leq -2$, is an
$I_4$-orbit. 

\item $E^{\hyp\open}_n$, $n > 0$, $E^{\hyp\hyp}_n$, $n \geq 2$, is an
$I_4$-orbit. 

\item $O^{\hyp \open}_n$, $O^{\hyp\hyp}_n$, $n > 0$, is an $I_4$-orbit. 

\item $E^{\open}_n$, $n > 0$, decomposes into the
$I_4^{\rot}$-orbit

\begin{align*}
E^{\open \open}_n &:= [n, \C^{\times}t^{-n}+\C^{\times}t^{-n+1}+\C t^{-n+2}+\dots + \C t^{n-1}]
\cong \Gm^2 \times \Aff^{2n-2}
\end{align*}

and the $I_4$-orbit

\begin{align*}
E^{\open \hyp}_n &:= [n, \C^{\times}t^{-n}+\C t^{-n+2}+\dots + \C t^{n-1}] \cong \Gm \times \Aff^{2n-2}\;.
\end{align*}

\item $O^{\open}_n$, $n > 0$, decomposes into the
$I_4^{\rot}$-orbit

\begin{align*}
O^{\open \open}_n &:= [n, \C^{\times}t^{-n}+\C^{\times}t^{-n+1}+\C t^{-n+2}+\dots + \C t^n]^{\prime} \cong \Gm^2 \times \Aff^{2n-1}
\end{align*}

and the $I_4$-orbit

\begin{align*}
O^{\open \hyp}_n &:= [n, \C^{\times}t^{-n}+\C t^{-n+2}+\dots + \C t^n]^{\prime} \cong \Gm \times \Aff^{2n-1}\;.  
\end{align*}
\end{enumerate}
\end{Prop}

\begin{Rem}
In the proposition, all subvarieties are obviously $\Gm^{\rot}$-invariant, see Remark \autoref{Rem:Gmrotactionon[]primeand[]primeprime}. 
We have thus written $I_4$
instead of $I_4^{\rot}$ whenever this is possible, for precision. 
The arguments of the proof of the proposition in fact show that the $I_4^{\rot}$-orbit
$E_n^{\open\open}$, $n > 0$, decomposes 
as $E_n^{\open\open} = \bigsqcup_{\beta \in \C^{\times}} E_n^{\open\open\beta}$, where

\begin{align*}
& E_n^{\open\open\beta} := \\
& \{[n, p_{-n} t^{-n}+p_{-n+1} t^{-n+1}+\C t^{-n+2}+\dots
+\C t^{n-1}] \;\vert\; p_{-n}, p_{-n+1} \in \C^{\times},\; p_{-n} = \beta p_{-n+1}\} \\
& \cong \Gm \times \Aff^{2n-2}
\end{align*}

is an $I_4$-orbit. Similarly we have for $n >0$
$O_n^{\open\open} = \bigsqcup_{\beta \in \C^{\times}} O_n^{\open\open\beta}$,

where 

\begin{align*}
& O_n^{\open\open \beta} := \\
& \{[n, p_{-n} t^{-n}+p_{-n+1} t^{-n+1}+\C t^{-n+2}+\dots
+\C t^n ]^{\prime} \;\vert\; p_{-n}, p_{-n+1} \in \C^{\times},\; p_{-n} = \beta p_{-n+1}\} \\
& \cong \Gm \times \Aff^{2n-1}
\end{align*}

is an $I_4$-orbit. 
\end{Rem}

\begin{Rem} By the proposition, the list of $I_4^{\rot}$-orbits in $\SL_2((t))/I$ with their distinguished points is

\begin{align*}
& E_0 = [0,0] \\
& E_n^{\open} \ni [-n-1, t^n]^{\prime},\ n < 0 \\
& E_n^{\hyp} \ni [n,0],\ n < 0\\
& O_n^{\open} \ni \begin{cases} [0,1]^{\prime} & n=0 \\ [-n-1,t^n] & n < 0 \end{cases},\ n \leq 0\\
& O_n^{\hyp} \ni [n,0]^{\prime},\ n \leq 0 \\
& E_n^{\hyp\open} \ni [n,t^{-n+1}],\ n > 0\\
& E_n^{\hyp\hyp} \ni [n,0],\ n > 0\\
& E_n^{\open\open} \ni [n,t^{-n}+t^{-n+1}],\ n > 0\\
& E_n^{\open\hyp} \ni [n, t^{-n}],\ n > 0\\
& O_n^{\hyp\open} \ni [n, t^{-n+1}]^{\prime},\ n > 0\\
& O_n^{\hyp\hyp} \ni [n,0]^{\prime},\ n > 0\\
& O_n^{\open\open} \ni [n, t^{-n}+t^{-n+1}]^{\prime},\ n > 0 \\
& O_n^{\open\hyp} \ni [n, t^{-n}]^{\prime},\ n > 0\;.
\end{align*}

The point orbits
are $E_0 = [0,0]$, $O_0^{\hyp} = [0,0]^{\prime}$, $O_{-1}^{\hyp}=[-1,0]^{\prime}$, $E_1^{\hyp\hyp} = [1,0]$. 
\end{Rem}

\begin{Rem}
Since $\dot{s_1} I_4 \dot{s_1}^{-1} = I_4$ it follows that if 
$O$ is an $I_4^{\rot}$-orbit in $\SL_2((t))/I$, then so is $\dot{s_1}O$.
From Remark \autoref{Rem:involutiondots} we find that the involution $\dot{s_1}$ of 
$\SL_2((t))/I$ restricts to isomorphisms

\begin{align*}
& E_0=[0,0] \mapsto O_0^{\hyp} = [0,0]^{\prime}\\
& E_n^{\open} \xrightarrow{\cong} O_{-n}^{\hyp \open},\ n < 0 \\
& E_n^{\hyp} \xrightarrow{\cong} O_{-n}^{\hyp\hyp},\ n < 0 \\
& O_0^{\open} \xrightarrow{\cong} O_0^{\open} \\
& O_n^{\open} \xrightarrow{\cong} E_{-n}^{\hyp\open},\ n < 0 \\
& O_n^{\hyp} \xrightarrow{\cong} E_{-n}^{\hyp\hyp},\ n < 0 \\
& E_n^{\open\open} \xrightarrow{\cong} E_n^{\open\open},\ n > 0 \\
& E_n^{\open\hyp} \xrightarrow{\cong} E_n^{\open\hyp},\ n > 0 \\
& O_n^{\open\open} \xrightarrow{\cong} O_n^{\open\open},\ n > 0 \\
& O_n^{\open\hyp} \xrightarrow{\cong} O_n^{\open\hyp},\ n > 0\;.
\end{align*}

\end{Rem}

\begin{proof}
\noindent

(1) follows from Proposition \autoref{Prop:I2orbits}(3). \\

(2)
\emph{$O_0^{\open}$ is an $I_4$-orbit.} 
By Proposition \autoref{Prop:I1orbits}(4) $O_0^{\open}$  is $I_4$-invariant.  
We have $O_0^{\open} = [0, \C^{\times}]^{\prime}$ and hence acting with $\begin{pmatrix} \alpha & 0 \\ 0 & \alpha^{-1}\end{pmatrix} \in I_4$, 
$\alpha \in \C^{\times}$, 
we see that it is an $I_4$-orbit. \par
 \emph{$O_n^{\open}$, $n < 0$, $O_n^{\hyp}$, $n \leq -2$, is an $I_4$-orbit.} 
 This follows from Proposition \autoref{Prop:I2orbits}(4). \par
(3) follows from Proposition \autoref{Prop:I3orbits}(3).\par
(4) follows from Proposition \autoref{Prop:I3orbits}(4). \par
(5) 
 We have $E_n^{\open} = E_n^{\open\open}\sqcup E_n^{\open\hyp}$
 by definition of $E_n^{\open}$. \par
\emph{$E_n^{\open\open}$ is an $I_4^{\rot}$-orbit.} First, we show $I_4 [n, t^{-n}+t^{-n+1}] \subseteq E_n^{\open\open}$. 

We have for $\begin{pmatrix} \alpha(1+t^2a) & t^2b \\ t^2c & \alpha^{-1}(1+t^2d)\end{pmatrix} \in I_4$

\begin{align}\label{eq:generaleltofI4on[n,t^{-n}+t^{-n+1}]^prime}
\begin{split}
& \begin{pmatrix} \alpha(1+t^2a) & t^2b \\ t^2c & \alpha^{-1}(1+t^2d)\end{pmatrix} \begin{pmatrix} t^n & t^{-n}+t^{-n+1} \\ 0 & t^{-n}\end{pmatrix}I = \\
& \begin{pmatrix} \alpha(1+t^2a)t^n  & \alpha(1+t^2a)(t^{-n}+t^{-n+1})+t^{-n+2}b \\ t^{n+2}c & t^2c(t^{-n}+t^{-n+1})+\alpha^{-1}(1+t^2 d)t^{-n}\end{pmatrix} I\;.
\end{split}
\end{align}

We have

\begin{align*}
\frac{\widetilde{c}}{\widetilde{d}} = \frac{t^{2n+2}c}{\alpha^{-1}+t^2(\alpha^{-1}d+c)+t^3c} \in t\C[[t]]
\end{align*}

and since 

\begin{align*}
\widetilde{d} = t^{-n} \widetilde{d}_{(0)} = t^{-n}(\alpha^{-1}+t^2(\alpha^{-1}d+c)+t^3c) 
\end{align*}

\eqref{eq:generaleltofI4on[n,t^{-n}+t^{-n+1}]^prime} equals

\begin{align}\label{eq:alpha^2t^(-n)+alpha^2t^(-n+1)prime}
\begin{pmatrix} t^n & \frac{\alpha(1+t^2a)(t^{-n}+t^{-n+1})+t^{-n+2}b}{\alpha^{-1}+t^2(\alpha^{-1}d+c)+t^3c}\\ 0 & t^{-n}\end{pmatrix}I\;.
\end{align}

We have

\begin{align*}
& \frac{\alpha(1+t^2a)(t^{-n}+t^{-n+1})+t^{-n+2}b}{\alpha^{-1}+t^2(\alpha^{-1}d+c)+t^3c} = \frac{\alpha^2 t^{-n}+\alpha^2 t^{-n+1} + \alpha(\alpha a+b)t^{-n+2}+\alpha^2 a t^{-n+3}}{1+t^2(d+\alpha c)+t^3\alpha c} \\
& \in \alpha^2t^{-n}+\alpha^2 t^{-n+1}+p_{-n+2}t^{-n+2}+\dots + p_{n-1} t^{n-1}+t^n \C[[t]]
\end{align*}

for some $p_{-n+2}, \dots, p_{n-1} \in \C$. Hence we see $I_4[n,t^{-n}+t^{-n+1}] \subseteq E_n^{\open\open}$
and hence $I_4^{\rot}[n,t^{-n}+t^{-n+1}] \subseteq E_n^{\open\open}$
since $E_n^{\open\open}$ is $\Gm^{\rot}$-invariant. 
Second, we show  $I_4^{\rot}[n,t^{-n}+t^{-n+1}] \supseteq E_n^{\open\open}$. 
Using the $\Gm^{\rot}$-action, see Remark \autoref{Rem:Gmrotactionon[]primeand[]primeprime},  on \eqref{eq:alpha^2t^(-n)+alpha^2t^(-n+1)prime} we have for $\gamma \in \C^{\times}$

\begin{align*}
& \gamma \cdot \begin{pmatrix} t^n & \alpha^2 t^{-n}+\alpha^2 t^{-n+1}+p_{-n+2}t^{-n+2}+\dots + p_{n-1}t^{n-1}\\ 0 & t^{-n}\end{pmatrix}I = \\
&  \begin{pmatrix} t^n & \alpha^2 t^{-n}+\gamma \alpha^2 t^{-n+1}+\gamma^2 p_{-n+2}t^{-n+2}
+\dots+\gamma^{2n-1}p_{n-1}t^{n-1} \\ 0 & t^{-n}\end{pmatrix}I\;.
\end{align*}

Now using that $\begin{pmatrix} 1 & t^2b \\ 0 & 1\end{pmatrix} \in I_4$ for $b \in \C[[t]]$ and

\begin{align*}
\begin{pmatrix} 1 & t^2b \\ 0 & 1\end{pmatrix} [n, p] = [n, p+t^{-n+2}b]
\end{align*}

we can make the coefficients of $t^{-n+2}$, \dots, $t^{n-1}$ arbitrary elements of $\C$. 
Thus we have shown $I_4^{\rot} [n,t^{-n}+t^{-n+1}] \supseteq E_n^{\open\open}$. \par

\emph{$E_n^{\open \hyp}$ is an $I_4$-orbit.} Since $E_n^{\open} = E_n^{\open\open}\sqcup E_n^{\open\hyp}$ and $E_n^{\open\open}$ are $I_4$-invariant, so is 
$E_n^{\open\hyp}$. Since 

\begin{align*}
I_4 [n, t^{-n}] \supseteq \begin{pmatrix} 1 & \C t^2+\dots + \C t^{2n-1} \\
0 & 1\end{pmatrix} [n, t^{-n}] = [n, t^{-n}+\C t^{-n+2}+\dots + \C t^{n-1}]
\end{align*}

and for $\alpha \in \C^{\times}$

\begin{align*}
\begin{pmatrix} \alpha & 0 \\ 0 & \alpha^{-1}\end{pmatrix} [n, t^{-n}+\C t^{-n+2}+\dots + \C t^{n-1}] = [n, \alpha^2 t^{-n}+\C t^{-n+2}+\dots + \C t^{n-1}]
\end{align*}

it also follows that $I_4 [n, t^{-n}] \supseteq E_n^{\open\hyp}$. \par

(6) is analogous to (5). 
\end{proof}

\bibliographystyle{alpha}
\bibliography{references2025}

\begin{thebibliography}{Eic20}

\bibitem[BD]{BD}
A.~Beilinson and V.~Drinfeld.
\newblock {\em Quantization of {Hitchin}'s integrable system and {Hecke}
  eigensheaves}.
\newblock Unpublished.

\bibitem[Eic16]{Eic16b}
C.~Eicher.
\newblock Relaxed highest weight modules from $\mathcal{D}$-modules on the
  {Kashiwara} flag scheme.
\newblock \href{https://arxiv.org/abs/1607.06342}{\nolinkurl{arXiv:1607.06342
  [math.RT]}}, 2016.

\bibitem[Eic20]{Eic20}
C.~Eicher.
\newblock Twisted $\mathcal{D}$-module extensions of local systems on a certain
  subvariety isomorphic to $\mathbb{G}_{\text{m}}^2$ of the affine flag variety
  of $\text{SL}_2$.
\newblock \href{https://arxiv.org/abs/2011.03764}{\nolinkurl{arXiv:2011.03764
  [math.AG]}}, 2020.

\end{thebibliography}
\end{document}